%% file: bk2.tex
\documentclass{amsart}
\input{settings}

\usepackage{graphics}
\usepackage{amsmath, amsfonts, amssymb}
\usepackage[foot]{amsaddr}
\usepackage[all]{xy}

\def\a{\alpha}

\def\g{\gamma}
\def\G{\Gamma}

\def\l{\lambda}

\newcommand{\mE}{\mathcal{E}}

\newcommand{\mP}{\mathcal{P}}

\newcommand{\mT}{\mathcal{T}}

\newcommand{\ff}{\mathfrak{f}}

\newcommand{\fm}{\mathfrak{m}}
\newcommand{\fM}{\mathfrak{M}}

\newcommand{\fp}{\mathfrak{p}}

\newcommand{\fq}{\mathfrak{q}}
\newcommand{\fQ}{\mathfrak{Q}}

\newcommand{\fS}{\mathfrak{S}}

\newcommand{\bfA}{\mathbf{A}}

\newcommand{\bfC}{\mathbf{C}}

\newcommand{\bfF}{\mathbf{F}}

\newcommand{\bfQ}{\mathbf{Q}} 
\newcommand{\bfR}{\mathbf{R}} 
 
\newcommand{\bfT}{\mathbf{T}}

\newcommand{\bfZ}{\mathbf{Z}}

\newcommand{\Oo}{\mathcal{O}}

\newcommand{\Oe}{\mathcal{O}_E}
\newcommand{\OF}{\mathcal{O}_F}

\newcommand{\OK}{\mathcal{O}_K}

\newcommand{\AF}{\mathbf{A}_F}
\newcommand{\AK}{\mathbf{A}_K}

\newcommand{\AQ}{\mathbf{A}}

\newcommand{\AQf}{\mathbf{A}_{\textup{f}}}

\newcommand{\OFq}{\mathcal{O}_{F,\mathfrak{q}}}

\newcommand{\tuf}{\textup{f}}

\newcommand{\tuint}{\textup{int}}

\newcommand{\tuord}{\textup{ord}}   

\newcommand{\tuss}{\textup{ss}}

\newcommand{\ov}{\overline}

\newcommand{\be}{\begin{equation}}  
\newcommand{\ee}{\end{equation}}
\newcommand{\bes}{\begin{equation*}}
\newcommand{\ees}{\end{equation*}}  

\newcommand{\bs}{\begin{split}}
\newcommand{\es}{\end{split}}
\newcommand{\bss}{\begin{split*}}   
\newcommand{\ess}{\end{split*}}

\newcommand{\bmat}{\left[ \begin{matrix}}
\newcommand{\emat}{\end{matrix} \right]}
\newcommand{\bsmat}{\left[ \begin{smallmatrix}}
\newcommand{\esmat}{\end{smallmatrix} \right]}

\newcommand{\bml}{\begin{multline}}
\newcommand{\eml}{\end{multline}}
\newcommand{\bmls}{\begin{multline*}}
\newcommand{\emls}{\end{multline*}}

\DeclareMathOperator{\ad}{ad}

\DeclareMathOperator{\Cl}{Cl}

\DeclareMathOperator{\End}{End}

\DeclareMathOperator{\Frob}{Frob}
\DeclareMathOperator{\Gal}{Gal}
\DeclareMathOperator{\GL}{GL}

\DeclareMathOperator{\Hom}{Hom}

\DeclareMathOperator{\Res}{Res}

\DeclareMathOperator{\rk}{rk}

\DeclareMathOperator{\val}{val}

\def\AQf{\mathbf{A}_{\textup{f}}}

\newcommand{\hs}{\hspace{2pt}}      
\newcommand{\hf}{\hspace{5pt}}  

\newcommand{\iy}{\infty}

\newcommand{\tr}{\textup{tr}\hspace{2pt}}

\begin{document}
\title[A deformation problem]{A deformation problem for Galois 
representations over imaginary 
quadratic fields}
\author{Tobias Berger$^1$ \and
Krzysztof Klosin$^2$}
\address{$^1$University of Cambridge, Department of 
Pure Mathematics and
Mathematical Statistics, Centre for
Mathematical Sciences, Cambridge, CB3 0WB, United Kingdom.}
\address{$^2$
University
of Utah,
Department of Mathematics, LCB,
155 S 1400 E RM 233, Salt Lake City, UT, 84112-0090,
USA.
}                     
%
%

%
\maketitle
\begin{abstract}
We prove the modularity of minimally ramified ordinary residually 
reducible $p$-adic Galois representations of an imaginary quadratic
field $F$ under certain assumptions. We first exhibit conditions
under which the residual representation is unique up to isomorphism.
Then we prove the existence of deformations arising from cuspforms
on $\GL_2(\AF)$ via the Galois representations constructed by Taylor
\textit{et al.} 
We establish a sufficient condition (in terms of the non-existence of
certain field extensions which in many cases can be reduced to a condition
on an $L$-value) for the universal deformation ring to be a discrete 
valuation ring and in 
that
case we prove an $R=T$ theorem.
We also study reducible deformations and show that no
minimal characteristic $0$ reducible deformation exists. 
\end{abstract}
\maketitle

\section{Introduction} \label{Introduction}

Starting with the work of Wiles (\cite{Wiles95},
\cite{TaylorWiles95}) there has been a lot of progress in recent
years on modularity results for two-dimensional $p$-adic Galois
representations of totally real fields (see e.g. \cite{BCDT},
\cite{SkinnerWiles97}, \cite{Fujiwara99}, \cite{SkinnerWiles99},
\cite{SkinnerWiles01}, \cite{Taylor02}, \cite{Kisin07}). The goal of
this paper is to prove such a result for imaginary quadratic fields,
a case that requires new techniques since the associated symmetric
space has no complex structure.

Let $F\neq \bfQ(\sqrt{-1}), \bfQ(\sqrt{-3})$ be an imaginary
quadratic field of discriminant $d_F$. Under certain assumptions we prove 
an ``$R=T$"
theorem for residually reducible two-dimensional representations of
the absolute Galois group of $F$. We pin down conditions (similar to
\cite{SkinnerWiles97}, where an analogous problem is treated for
representations of $\Gal(\ov{\bfQ}/\bfQ)$) that determine our
residual representation up to isomorphism and then study its minimal
ordinary deformations. Modular deformations are constructed using
the congruences involving Eisenstein cohomology classes of
\cite{B09} and the result of Taylor on associating
Galois representations to certain cuspidal automorphic
representations over imaginary quadratic fields (using the
improvements of \cite{BergerHarcos07}). The approach of
\cite{SkinnerWiles97} to prove the isomorphism between universal
deformation ring and Hecke algebra fails in our case because of the
non-existence of an ordinary reducible characteristic 0 deformation.
This failure, however, allows under an additional assumption to show
(using the method of \cite{BellaicheChenevier06}) that the
Eisenstein deformation ring is a discrete valuation ring. As in
\cite{Calegari06} it is then easy to deduce an ``$R=T$" theorem.

To give a more precise account, let $c$ be the non-trivial
automorphism of $F$, and let $p>3$ be a prime split in the extension
$F/\bfQ$. Fix embeddings $F \hookrightarrow \overline{\bfQ}
\hookrightarrow \overline{\bfQ}_p\hookrightarrow \bfC$. Let
$F_{\Sigma}$ be the maximal extension of $F$ unramified outside a
finite set of places $\Sigma$. Suppose $\bfF$ is a finite field of
characteristic $p$ and that $\chi_0:{\rm Gal}(F_{\Sigma}/F) \to
\bfF^{\times}$ is an anticyclotomic character ramified at the places
dividing $p$. Suppose also that $\rho_0:{\rm Gal}(F_{\Sigma}/F) \to
{\rm GL}_2(\bfF)$ is a continuous representation of the form
$$\rho_0=\begin{pmatrix} 1&*\\0&\chi_0
\end{pmatrix}$$ and having scalar centralizer. Under certain
conditions on $\chi_0$ and $\Sigma$ we show that $\rho_0$ is unique
up to isomorphism  (see Section \ref{Uniqueness of a certain
residual Galois representation}) and we fix a particular choice.
This setup is similar to that of \cite{SkinnerWiles97}. Note that,
as explained in Remark~\ref{specialbasis}, under our conditions $\rho_0$
does not arise
as the restriction of a representation of ${\rm
Gal}(\ov{\bfQ}/\bfQ)$.

Following Mazur \cite{Mazur97} we study ordinary deformations of
$\rho_0$. Let $\mathcal{O}$ be a local complete Noetherian ring with
residue field $\bfF$. An $\mathcal{O}$-deformation of $\rho_0$ is a
local complete Noetherian $\mathcal{O}$-algebra $A$ with residue
field $\bfF$ and maximal ideal $\mathfrak{m}_A$ together with an
equivalence class of continuous representations $\rho:{\rm
Gal}(F_{\Sigma}/F) \to {\rm GL}_2(A)$ satisfying $\rho_0=\rho
\mod{\mathfrak{m}_A}$. An \textit{ordinary} deformation (see, for
example, the definition in \cite{Weston05}) is a deformation that
satisfies
$$\rho|_{D_{\fq}}\cong
\begin{pmatrix} \chi_1^{({\fq})}&*\\0&\chi_2^{({\fq})}
\end{pmatrix}$$ for ${\fq} \mid p$, where
$\chi_i^{({\fq})}|_{I_{\fq}}=\tau_i^{({\fq})}
\epsilon^{k_i^{({\fq})}}$ with $k_1^{({\fq})} \geq k_2^{({\fq})}$,
$\epsilon$ is the $p$-adic cyclotomic character, and
$\tau_i^{({\fq})}$ are some finite order characters. Here $D_{\fq}$
and $I_{\fq}$ denote the decomposition group and the inertia group
of ${\fq} \mid p$, corresponding to $F \hookrightarrow
\overline{\bfQ}_p$ or the conjugate embedding, respectively.


To exhibit modular deformations we apply the cohomological
congruences of \cite{B09} and the Galois
representations constructed by Taylor \textit{et al.} using a
strengthening of Taylor's result in \cite{BergerHarcos07}. We also
make use of a result of Urban \cite{Urban05} who proves that
$\rho_{\pi}|_{D_{\fq}}$ is ordinary at ${\fq} \mid p$ if $\pi$ is
ordinary at ${\fq}$. We show that these results imply that there is
an $\mathcal{O}$-algebra surjection \begin{equation} \label{RsurjT}
R \twoheadrightarrow T,\end{equation} where $R$ is the universal
$\Sigma$-minimal deformation ring (cf. Definition \ref{sigmamin})
and $T$ is a Hecke algebra acting on cuspidal automorphic forms of
${\rm GL}_2(\bfA_F)$ of weight 2 and fixed level.

As in \cite{Calegari06} we can deduce that the surjection
(\ref{RsurjT}) is, in fact, an isomorphism if $R$ is a discrete
valuation ring (see Theorem \ref{heckedef}).  Using the method of
\cite{BellaicheChenevier06} we prove in Proposition~\ref{prop5.9}
that the latter reduces to the non-existence of reducible
$\Sigma$-minimal deformations to ${\rm GL}_2(\Oo/\varpi^2 \Oo)$
\footnote{Here $\varpi$ denotes a uniformizer of $\Oo$.}. 
We then show (Theorem~\ref{prop5.10}) that this last property can 
often be
deduced from a condition on the $L$-value at 1 of a Hecke character
of infinity type $z/\ov{z}$ which is related to $\chi_0$.
Finally
we combine these results in Theorem~\ref{cor5.8} to prove the
modularity of certain residually reducible $\Sigma$-minimal ${\rm
Gal}(\ov{\bfQ}/F)$-representations. For an explicit numerical
example where we can verify the conditions of Theorem~\ref{cor5.8}
see Example \ref{example1}.

To demonstrate our modularity result we give here the following
special case:

\begin{thm} Assume $\# \Cl_F=1$, that $p$ does not 
divide the class 
number of the ray class field of
  $F$ of conductor $p$, and that any prime $q \mid d_F$
satisfies $q \not\equiv \pm 1$ \textup{(mod} $p$\textup{)}. Let $\fp$ be 
the prime of $F$ over $(p)$ corresponding to the embedding $F 
\hookrightarrow \ov{\bfQ} \hookrightarrow \ov{\bfQ}_p$ that we have 
fixed. Let $\tau$ be the unramified Hecke character of infinity type
$\tau_{\infty}(z)=z/\ov{z}$ 
and let $\tau_{\fp}:{\rm Gal}(\overline \bfQ/F) \to
\bfZ_p^{\times}$ be the associated $p$-adic Galois character. Assume that
$\val_p(L^{\tuint}(1, \tau))=1$.\footnote{For definitions see section 
\ref{Notation and terminology}.}

Let $\rho: {\rm
Gal}(F_{\Sigma}/F) \rightarrow \GL_2(\ov{\bfQ}_p)$ be a continuous
irreducible representation that is ordinary at all places $\fq \mid
p$. Suppose $\ov{\rho}^{\tuss} \cong 1 \oplus \ov{\tau}_{\fp}$. If the
following conditions are satisfied:

\begin{enumerate}
  \item $\Sigma \supset \{{\fq} \mid p d_F\}$,
  \item if $\fq \in \Sigma, \fq \nmid p$, then $\overline\tau_{\fp} 
(\Frob_{\fq}) \neq \pm \# \OF/\fq$ as elements of $\bfF_p$,
  \item ${\rm det}(\rho)=\tau_{\fp}$,
  \item $\rho$ is $\Sigma$-minimal,
\end{enumerate}
then $\rho$ is isomorphic to the Galois representation associated to
a cuspform of ${\rm GL}_2(\bfA_F)$ of weight 2, twisted by the
$p$-adic Galois character associated to a Hecke character of
infinity type $z$.
\end{thm}

We also study the existence of reducible deformations (see Section
\ref{A reducible deformation of rho_0}). In contrast to the
situation in \cite{SkinnerWiles97} there exists no reducible
$\Sigma$-minimal $\mathcal{O}$-deformation in our case, only a
nearly ordinary (in the sense of Tilouine \cite{Tilouine96})
reducible deformation which is, however, not de Rham at one of the
places above $p$. This means that the method of
\cite{SkinnerWiles97} to prove $R=T$ via the numerical criterion of
Wiles and Lenstra \cite{Lenstra95}, \cite{Wiles95} cannot be implemented 
despite 
having all the
ingredients on the Hecke side (i.e., a lower bound on the congruence
module measuring congruences between cuspforms and Eisenstein
series).

 The assumption on $\chi_0$ being anticyclotomic could be relaxed but is useful
 both for proving
the uniqueness of $\rho_0$ and to construct the modular
deformations, and is related to a condition on the central character
in Taylor's result on associating Galois representation to
cuspforms. The restrictions in Definition~\ref{adm} on the places
contained in $\Sigma$ and on the class group of the splitting field
of $\chi_0$ are similar to those of \cite{SkinnerWiles97} and are
essential for the uniqueness of $\rho_0$. Our methods do not allow
to go beyond the $\Sigma$-minimal case (to achieve
that in the $\bfQ$-case \cite{SkinnerWiles97} use
Proposition 1 of \cite{TaylorWiles95}, but its
analogue fails for imaginary quadratic
fields) or treat residually irreducible Galois representations.
To
complement our study of the \textit{absolute} deformation problem of
a residually reducible Galois representation the reader is referred
to the analysis of the nearly ordinary \textit{relative} deformation
problem in \cite{CalegariMazur07}.

\section{Notation and terminology}\label{Notation and terminology}

\subsection{Galois groups} \label{Galois groups} Let $F$ be an imaginary
quadratic
extension of $\bfQ$ of discriminant $d_F \neq 3,4$
and $p>3$ a rational
prime which splits in $F$. Fix a prime $\fp$ of $F$ lying over $(p)$ and
denote the other prime of $F$ over $(p)$ by $\ov{\fp}$.
Let $\Cl_F$ denote the class group of
$F$.
We assume that $p \nmid \# \Cl_F$ and that any prime $q \mid d_F$
satisfies $q \not\equiv \pm 1 \pmod{p}.$

For a field $K$ write $G_K$ for the Galois group $\Gal(\ov{K}/K)$. If
$K\supset F$ is a number field, $\OK$ will denote its ring of integers. If
$K$ is a finite extension of $\bfQ_{\ell}$ for some rational prime $\ell$,
we write $\OK$ (respectively $\varpi_K$, and $\bfF_K$) for the ring of
integers of $K$ (respectively for a uniformizer of $K$, and $\OK/\varpi_K
\OK$). If $\fq$ is a place of $K$, we write $K_{\fq}$ for the completion
of $K$ with respect to the absolute value $|\cdot|_{\fq}$ determined by
$\fq$ and set $\Oo_{K,\fq} = \Oo_{K_{\fq}}$ (if $\fq$ is archimedean, we
set $\Oo_{K,\fq} = K_{\fq}$). We also write $\varpi_{\fq}$ for a
uniformizer of $K_{\fq}$.

Fix once and for all compatible embeddings $i_{\fq}:\ov{F}
\hookrightarrow \ov{F}_{\fq}$ and $\ov{F}_{\fq} \hookrightarrow
\bfC$, for every prime $\fq$ of $F$, so we will often regard
elements of $\ov{F}_{\fq}$ as complex numbers without explicitly
mentioning it. If $\fq$ is a place of $K \subset \overline F$,
we always regard $K_{\fq}$ as a subfield of $\ov{F}_{\fq \cap \OF}$
as determined by the embedding $i_{\fq \cap \OF}$. This also allows
us to identify $G_{K_{\fq}}$ with the decomposition group
$D_{\fQ}\subset G_K$ of a prime $\fQ$ of
the ring of integers $\Oo_{\ov{F}}$
of $\ov{F}$. We will denote that
decomposition group by $D_{\fq}$. Abusing notation somewhat we will
denote the image of $D_{\fq}$ in any quotient of $G_K$ also by
$D_{\fq}$. We write $I_{\fq} \subset D_{\fq}$ for the inertia group.

Let $\Sigma$ be a finite set of places of $K$. Then $K_{\Sigma}$ will
denote the maximal Galois extension of $K$ unramified outside the primes
in $\Sigma$. We also write $G_{\Sigma}$ for $G_{F_{\Sigma}}$.

For a positive integer $n$, denote by $\mu_n$ the group of $n$-th
roots of unity. If $K$ is a number field we set $K'=K(\mu_p)$.
Let $\omega_{K,p}$ denote the character
giving the action of $\Gal(K'/K)$ on $\mu_p$. Let $\Cl_{K,p}$ denote
the Sylow-$p$-subgroup of the quotient of the class group $\Cl_{K'}$
of $K'$ corresponding (by Class Field Theory) to the quotient
$I_{K'}/P_{K'} \mP$, where $I_{K'}$ is the group of fractional
ideals of $K'$, $P_{K'}$ the subgroup of principal ideals and $\mP$
the subgroup of $I_{K'}$ generated by the primes of $K'$ lying over
$p$. We will write $\Cl_{K,p}^{\omega}$ for the $\omega_{K,p}$-part
of $\Cl_{K,p}$.

\subsection{Hecke characters} For a number field $K$, denote by $\AK$
the ring of adeles of $K$ and set $\AQ = \bfA_{\bfQ}$. By a
\textit{Hecke character} of $K$ we mean a continuous homomorphism
$$\lambda: K^{\times} \setminus \AK^{\times} \rightarrow
\bfC^{\times}.$$ For a place $\fq$ of $K$ write $\l^{(\fq)}$ for the
restriction of $\l$ to $K_{\fq}$ and $\l^{(\iy)}$ for the
restriction of $\l$ to $\prod_{\fq \mid \iy} K_{\fq}$. The latter
will be called the \textit{infinity type} of $\lambda$. We also
usually write $\l(\varpi_{\fq})$ to mean $\l^{(\fq)}(\varpi_{\fq})$.
Given $\lambda$ there exists a unique ideal $\ff_{\lambda}$ of $K$
with the property that $\lambda^{(\fq)}(x)=1$ for every finite place
$\fq$ of $K$ and $x \in \Oo_{K,\fq}^{\times}$ such that $x-1
\in \ff_{\lambda} \Oo_{K,\fq}$. The ideal $\ff_{\lambda}$ is called
the \textit{conductor} of $\lambda$. If $K=F$, there is only one
archimedean place, which we will simply denote by $\iy$. For a Hecke
character $\lambda$ of $F$, one has $\lambda^{(\iy)} (z) = z^m
\ov{z}^n$ with $m,n \in \bfR$. If $m,n \in \bfZ$, we say that
$\lambda$ is of type ($A_0$). We always assume that our Hecke
characters are of type ($A_0$). Write $L(s,\lambda)$ for the Hecke
$L$-function of $\lambda$. Let $\lambda$ be a Hecke character of
infinity type $z^a \left( \frac{z}{\overline z} \right )^{b}$ with
conductor prime to $p$. Assume $a,b \in \bfZ$ and $a>0$ and ${b}
\geq 0$. Put
$$L^{\mathrm{alg}}(0,\lambda):= \Omega^{-a-2{b}} \left(
\frac{2\pi}{\sqrt{d_F}}\right)^{b} \Gamma(a+{b})\cdot
L(0,\lambda),$$ where $\Omega$ is a complex period. In most cases,
this normalization is integral, i.e., lies in the integer ring of a
finite extension of $F_{\mathfrak{p}}$. See \cite{Berger08} Theorem 3
for the exact statement. Put $$L^{\rm int}(0, \lambda)=
  \begin{cases}
    L^{\mathrm{alg}}(0,\lambda) & \text{ if } {\rm val}_p (L^{\mathrm{alg}}(0,\lambda)) \geq 0 \\
    1 & \text{otherwise}.
  \end{cases}
$$

For $z \in \bfC$ we write $\ov{z}$ for the complex conjugate of $z$. The
action of complex conjugation extends to an automorphism of $\AF^{\times}$
and we will write $\ov{x}$ for the image of $x \in \AF^{\times}$ under
that automorphism.

For a Hecke character $\lambda$ of $F$, we denote by $\lambda^c$ the Hecke
character of $F$ defined by $\lambda^c(x) = \lambda(\ov{x})$.

\subsection{Galois representations} \label{Galois representations} For a
field $K$ and a topological
field
$E$, by a
\textit{Galois representation} we mean a continuous homomorphism $\rho:
G_K
\rightarrow \GL_n(E)$. If $n=1$ we usually refer to $\rho$ as a
\textit{Galois
character}. We write $K(\rho)$ for the fixed field of $\ker \rho$ and
call it the \textit{splitting field of $\rho$}. If $K$ is a number
field and $\fq$ is a finite prime of $K$ with inertia group $I_{\fq}$ we
say
that $\rho$ is unramified at $\fq$ if $\rho|_{I_{\fq}} = 1$.

Let $E$
be a finite extension of $\bfQ_p$. Every Galois
representation $\rho : G_K \rightarrow \GL_n(E)$ can be conjugated (by an
element $M \in \GL_n(E)$) to a
representation $\rho_M: G_K \rightarrow \GL_n(\Oe)$. We denote by
$\ov{\rho}_M: G_F \rightarrow \GL_n(\bfF_E)$ its reduction modulo
$\varpi_E\Oe$. It is sometimes called a \textit{residual representation}
of $\rho$. The isomorphism class of its semisimplification
$\ov{\rho}_M^{\tuss}$ is
independent of the choice of $M$ and we simply write $\ov{\rho}^{\tuss}$.



Let $\epsilon: G_F
\rightarrow \bfZ_p^{\times}$ denote the $p$-adic cyclotomic character.
For any subgroup $G \subset
G_F$ we will also write $\epsilon$ for $\epsilon|_{G}$. Our
convention is that the Hodge-Tate weight of $\epsilon$ at $\fp$ is 1.

Let $\lambda$ be a Hecke character of $F$ of type ($A_0$). We define
(following Weil) a $\fp$-adic Galois character $$\lambda_{\fp}: G_F
\rightarrow \ov{F}_{\fp}^{\times}$$
associated to $\l$ by the following rule: For a finite place $\fq
\nmid p \ff_{\lambda}$ of $F$, put $\l_{\fp}(\Frob_{\fq}) = i_{\fp}
(i_{\iy}^{-1}(\l(\varpi_{\fq})))$ where $\Frob_{\fq}$ denotes the
\textit{arithmetic} Frobenius at $\fq$. It takes values in the
integer ring of a finite extension of $F_{\fp}$.

\subsection{Automorphic representations of $\AF$ and their Galois
representations} Set $G= \Res_{F/\bfQ} \GL_2$. For
$K_{\tuf}=\prod_{\fq\nmid \iy}
K_{\fq}$ an open
compact subgroup of $G(\AQf)$, denote by $S_2(K_{\tuf})$ the space of
cuspidal
automorphic forms of $G(\AQ)$ of weight 2, right-invariant under
$K_{\tuf}$ (for more details see Section 3.1 of \cite{Urban95}).
For $\psi$ a finite order Hecke character write $S_2(K_{\tuf}, \psi)$ for
the forms with central character $\psi$. This is isomorphic as a
$G(\AQf)$-module to $\bigoplus \pi_{\tuf}^{K_{\tuf}}$ for automorphic
representations $\pi$ of certain infinity type (see Theorem \ref{attach1}
below)
with central character $\psi$. Here $\pi_{\tuf}$ denotes the restriction
of $\pi$ to $\GL_2(\AQf)$ and $\pi_{\tuf}^{K_{\tuf}}$ stands for the
$K_{\tuf}$-invariants.

For $g \in G(\AQf)$ we have the usual Hecke
action of $[K_{\tuf} g K_{\tuf}]$ on $S_2(K_{\tuf})$ and $S_2(K_{\tuf},
\psi)$. For primes $\fq$ with $K_{\fq} = \GL_2(\OFq)$ we define $T_{\fq}
= [K_{\tuf} \bmat \varpi_{\fq} \\ & 1 \emat K_{\tuf}] $.

Combining the work of Taylor, Harris, and Soudry with results of
Friedberg-Hoffstein and Laumon/Weissauer, one can show the following
(see \cite{BergerHarcos07} for general case of cuspforms of weight
$k$):
\begin{thm}[\cite{BergerHarcos07} Theorem 1.1] \label{attach1} Given a 
cuspidal automorphic representation
$\pi$ of
$\GL_2(\AF)$ with $\pi_{\iy}$ isomorphic to the principal series
representation corresponding to $$\bmat t_1 & * \\ & t_2 \emat \mapsto
\left(\frac{t_1}{|t_1|} \right) \left(\frac{|t_2|}{t_2}\right)$$ and
cyclotomic central character
$\psi$ (i.e., $\psi^c = \psi$), let $\Sigma_{\pi}$ denote the set
consisting of the places of $F$ lying above $p$, the primes where $\pi$ or
$\pi^c$ is ramified, and the primes ramified in $F/\bfQ$.

Then there exists a finite extension $E$ of $F_{\fp}$ and a Galois
representation $$\rho_{\pi}: G_F \rightarrow
\GL_2(E)$$ such that if $\fq \not\in \Sigma_{\pi}$, then
$\rho_{\pi}$ is unramified at $\fq$ and the characteristic polynomial of
$\rho_{\pi}(\Frob_{\fq})$ is $x^2-a_{\fq}(\pi)x + \psi(\varpi_{\fq}) (\#
\OF/\fq),$ where $a_{\fq}(\pi)$ is the Hecke eigenvalue corresponding to
$T_{\fq}$.
Moreover,
$\rho_{\pi}$ is absolutely irreducible. \end{thm}

\begin{rem} Taylor has some additional technical assumption in
\cite{Taylor94} and only showed the equality of Hecke and Frobenius
polynomial outside a set of places of zero density. Conjecture 3.2
in \cite{CalegariDunfield06} describes a conjectural extension of
Taylor's theorem.
\end{rem}

Urban studied in \cite{Urban98} the case of ordinary automorphic
representations $\pi$, and together with results in \cite{Urban05}
on the Galois representations attached to ordinary Siegel modular
forms showed:

\begin{thm} [Corollary 2 of \cite{Urban05}] \label{Urbanordinary}
Let
$\fq$ be a prime of $F$ lying over $p$. If
$\pi$
is unramified at $\fq$ and ordinary at $\fq$, i.e., $|a_{\fq}(\pi)|_{\fq}
= 1$, then the Galois representation $\rho_{\pi}$ is ordinary at $\fq$,
i.e.,
$$\rho_{\pi}|_{D_{\fq}} \cong \bmat \Psi_1 & *\\ &
\Psi_2
\emat,$$ where $\Psi_2|_{I_{\fq}} = 1$ and $\Psi_1|_{I_{\fq}} = \det
\rho_{\pi}|_{I_{\fq}} = \epsilon.$ \end{thm}

\begin{definition} \label{modular1} Let $E$ be a finite extension of
$F_{\fp}$ and $\rho:
G_F \rightarrow \GL_2(E)$ a Galois representation. We say that $\rho$ is
\emph{modular} if there exists a cuspidal automorphic representation
$\pi$ as in Theorem \ref{attach1}, such that $\rho \cong \rho_{\pi}$
(possibly after enlarging $E$). \end{definition}

From now on we fix a finite extension $E$ of $F_{\fp}$ which we assume to
be sufficiently large (see section \ref{defining
rho_0} and Remarks \ref{r4.5} and \ref{OtoO'}, where 
this condition is 
made more precise). To 
simplify notation we put 
$\Oo:= 
\Oe$,
$\bfF=\bfF_E$ and $\varpi=\varpi_E$.

\section{Uniqueness of a certain residual Galois representation}
\label{Uniqueness of a certain residual Galois representation}

In this section we study residual Galois representations $\rho_0 : G_F
\rightarrow \GL_2(\bfF)$ of the form
$$\rho_0 = \bmat 1&*\\ & \chi_0\emat $$ having scalar centralizer for a
certain class of characters $\chi_0$ (cf. Definition \ref{adm}). We
show that for a fixed $\chi_0$ there exists at most one such
representation up to isomorphism (Theorem \ref{essuni}). In Section
\ref{Irreducible modular deformations of rho_0} we show that there
indeed exists one provided that $\val_p(L(0, \phi))>0$ for a certain
Hecke character $\phi$ of $F$ such that the reduction of $\phi_{\fp}
\epsilon$ is $\chi_0$. Alternatively, one could invoke the
generalizations of Kummer's criterion to imaginary quadratic fields
(see e.g. \cite{CoatesWiles77}, \cite{Yager}, \cite{Hida82},
\cite{LozanoRobledo07}).

Let $\Sigma$ be a finite set of finite primes of $F$ containing the
primes lying over $p$ and let $\chi_0: G_{\Sigma}
\rightarrow \bfF^{\times}$ be a Galois character.

\begin{definition} \label{adm} We say that $\chi_0$
is $\Sigma$-\emph{admissible} if all of the following conditions
are satisfied:
\begin{enumerate}
\item $\chi_0$ is ramified at $\fp$;
\item if $\fq \in \Sigma$, then either $\chi_0$ is ramified at $\fq$ or
$\chi_0^{-1} (\Frob_{\fq}) \neq \# \OF/\fq$ (as elements of
$\bfF$);
\item $\chi_0$ is anticyclotomic, i.e., $\chi_0(c \sigma c) =
\chi_0(\sigma)^{-1}$ for every $\sigma \in G_{\Sigma}$ and $c$ the
generator of $\Gal(F/\bfQ)$; \item
$\Cl_{F(\chi_0),p}^{\omega}=0$ (cf. Section \ref{Galois groups});
\item The $\chi_0^{-1}$-eigenspace of the $p$-part of $\Cl_{F(\chi_0)}$ is
trivial. \end{enumerate} \end{definition}
Note that Conditions (1) and (3) of Definition \ref{adm} imply that
$\chi_0$ is also ramified at $\ov{\fp}$.
Fix $\tau \in I_{\fp}$ such that $\chi_0(\tau) \neq 1$. Let
$$\rho_0: G_{\Sigma} \rightarrow \GL_2(\bfF)$$ be a Galois
representation satisfying both of the following two conditions
\begin{description}
\item[(Red)]
$\rho_0
= \bmat 1 & * \\ & \chi_0 \emat$;
\item[(Sc)] $\rho_0$ has scalar
centralizer.
\end{description} We have the following tower of fields:
$F \subset F(\chi_0) \subset F(\rho_0)$.
Note
that
$p$ does not divide $[F(\chi_0):F]$, $F(\rho_0)/F(\chi_0)$ is an
abelian extension of exponent $p$,
hence $\Gal(F(\rho_0)/F(\chi_0))$ can
be
regarded as an $\bfF_p$-vector space
$V_0$ on which the group $G:= \Gal(F(\chi_0)/F)$ operates
$\bfF_p$-linearly by
conjugation and
thus defines a
representation $$r_0: G \rightarrow
\GL_{\bfF_p}(V_0),$$ which is isomorphic to
the irreducible
$\bfF_p$-representation associated with $\chi_0^{-1}$.

Let $L$ denote
the maximal abelian extension of $F(\chi_0)$ unramified outside the set
$\Sigma$ and such that $p$ annihilates
$\Gal(L/F(\chi_0))$. Then, as before, $V:= \Gal(L/F(\chi_0))$ is an
$\bfF_p$-vector space endowed with an $\bfF_p$-linear action of $G$, and
one has $$V \otimes_{\bfF_p} \ov{\bfF}_p \cong \bigoplus_{ \varphi \in
\Hom (G, \ov{\bfF}_p^{\times})} V^{\varphi},$$ where for a
$\bfZ_p[G]$-module $N$ and an $\ov{\bfF}_p$-valued character $\varphi$ of 
$G$, we write
\be \label{eigen1} N^{\varphi} = \{n
\in N \otimes_{\bfZ_p} \ov{\bfF}_p \mid \sigma n = \varphi (\sigma) n \hf
\textup{for every} \hs \sigma \in G\}.\ee Note that $V_0 \otimes_{\bfF_p}
\ov{\bfF}_p$ is a direct summand of $V^{\chi_0^{-1}}$.

\begin{thm}\label{essuni} If $\chi_0$ is $\Sigma$-admissible,
then $\dim_{\ov{\bfF}_p} V^{\chi_0^{-1}} = 1$.
\end{thm}

\begin{proof} Let $L_0$ be the maximal abelian extension of
$F(\chi_0)$ of exponent $p$ unramified outside the set $\Sigma$ and such
that $G$
acts on $\Gal(L_0/F(\chi_0))$ via the irreducible
$\bfF_p$-representation associated with $\chi_0^{-1}$. It is
enough to show that
$$\dim_{\ov{\bfF}_p}(\Gal(L_0/F(\chi_0))\otimes \ov{\bfF}_p) \leq
1.$$

Condition (2) of Definition \ref{adm} ensures that $L_0/F(\chi_0)$
is unramified outside the set $\{ \fp, \ov{\fp}\}$. Hence it is
enough to study the extensions $L/F(\chi_0)$ and $L_0/F(\chi_0)$
with $\Sigma = \{ \fp, \ov{\fp}\}$. For $\fp_0 \in \{\fp , \ov{\fp}
\}$ let $S_{\fp_0}$ be the set of primes of $F(\chi_0)$ lying over
$\fp_0$ and put $S_p:= S_{\fp} \cup S_{\ov{\fp}}$. Write $M$ for
$\prod_{\fq \in S_p} (1+ \fq)$ and $T$ for the torsion submodule of
$M$. By condition (5) of Definition \ref{adm} and Class Field Theory (see,
for example, Corollary 13.6 in \cite{Washingtonbook}) one has
$\Gal(L/F(\chi_0)) \cong (M/\ov{\mE}) \otimes \bfF_p$,
where $\ov{\mE}$ is the closure of $\mE$, the group of units
of the ring of integers of $F(\chi_0)$ which are congruent to 1
modulo every prime in $S_p$. Hence $\Gal(L_0/F(\chi_0))$ is a quotient of
$(M/\ov{\mE}) \otimes \bfF_p$.
On the other hand, using condition (3) of Definition \ref{adm}
one can show that
$\Gal(L_0/F(\chi_0))$ is a quotient of $(M/T)\otimes \bfF_p$. This
follows from the fact that $T$ is a product of the groups $\mu_p$; thus
$\chi_0$ being anticyclotomic by condition (3) of Definition
\ref{adm} cannot occur in
$T$. We will now study both $(M/T)\otimes \ov{\bfF}_p$ and $(M/\ov{\mE})
\otimes \ov{\bfF}_p$, beginning with the former one.

Let $G^{\vee}:= \Hom(G,
\ov{\bfF}_p^{\times})$. Since $G$ is
abelian,
$(M/T) \otimes \ov{\bfF}_p$ decomposes into a direct sum of $\ov{\bfF}_p
[G]$-modules $$(M/T)  \otimes
\ov{\bfF}_p =
\bigoplus_{\psi \in G^{\vee}} (M/T)^{\psi},$$
with $(M/T)^{\psi}$ defined as in (\ref{eigen1}).
Note
that we can refine this by writing $$M/T = \prod_{\fp_0\in \{\fp,
\ov{\fp}\}}
M_{\fp_0}/T_{\fp_0},$$ where $M_{\fp_0} = \prod_{\fq \in S_{\fp_0}}
(1+\fq)$ and $T_{\fp_0}$ is the torsion subgroup of $M_{\fp_0}$. Each
$M_{\fp_0}/T_{\fp_0}$ is $G$-stable.

\begin{lemma} \label{propdim} Let $\fp_0\in \{\fp,\ov{\fp}\}$. For every
$\psi \in G^{\vee}$, we have
$$\dim_{\ov{\bfF}_p} (M_{\fp_0}/T_{\fp_0})^{\psi} = 1.$$
\end{lemma}

\begin{proof} [Proof of Lemma \ref{propdim}] Note that to decompose
$(M_{\fp_0}/T_{\fp_0})
\otimes
\ov{\bfF}_p$ it is
enough to decompose $\prod_{\fq \in S_{\fp_0}} \fq \otimes \ov{\bfF}_p$,
since
$(1+\fq)/(\textup{torsion}) \cong \fq$ as $\bfZ_p[D_{\fq}]$-modules, where
$D_{\fq}$
denotes the decomposition group of $\fq$. It is not difficult to see that
$$\prod_{\fq \in S_{\fp_0}} \fq \otimes \ov{\bfF}_p \cong \bigoplus_{\phi
\in
G^{\vee}} \ov{\bfF}_p^{\phi},$$ where $\ov{\bfF}_p^{\phi}$ denotes the
one-dimensional $\ov{\bfF}_p$-vector space on which $G$ acts via $\phi$.
The Lemma follows easily. \end{proof}

Consider the exact sequence of $G$-modules \be \label{iota1}
\ov{\mE}\otimes\ov{\bfF}_p
\xrightarrow{\iota} M \otimes\ov{\bfF}_p
\rightarrow (M/\ov{\mE})
\otimes\ov{\bfF}_p \rightarrow 0.\ee

\begin{lemma} \label{Gras2} $\ker
\iota = 0$.
\end{lemma}

\begin{proof} [Proof of Lemma \ref{Gras2}] For a finitely generated
$\bfZ$-module $A$, write $\rk_p (A)$ for the dimension of the
$\bfF_p$-vector space $A/pA$. First note that since $\chi_0$ is
anticyclotomic, $\mu_p \not\subset F(\chi_0)$ and thus $\ov{\mE}$ is a
free $\bfZ_p$-module. By the Leopoldt conjecture (which is known for
$F(\chi_0)$ by a result of Brumer \cite{Brumer67}) we have
$\rk_{\bfZ_p} \ov{\mE} = r_2-1$ and since $\rk_{\bfZ_p} M = 2r_2$, we
can find a basis of the free part of $M$ (if the relative
ramification index $e$ of $\fp$ in the extension $F(\chi_0)/F$ is
smaller than $p-1$, then $M$ is free) such that the image of
$\ov{\mE}$ lands in the first $r_2-1$ $\bfZ_p$-factors of $M_{\rm
free} \cong \bfZ_p^{2r_2}$.
Under this
identification we have
$$M/\ov{\mE} = (\bfZ_p^{r_2-1}/\ov{\mE}) \times \bfZ_p^{r_2+1} \times
T.$$
 Note that $\ker \iota=0$ if and only if there
does not
exist $m \in M \setminus \ov{\mE}$ whose $p$-th power is in
$\ov{\mE}\setminus \{1\}$.
Hence
$\ker \iota=0$ if and only if $\rk_p(M/\ov{\mE})=r_2+1+d$, where
$d=\rk_p(T)$  is
the
number of
primes $\fq$ of $F(\chi_0)$ over $p$ such that $\mu_p \subset
F(\chi_0)_{\fq}$ (since $F(\chi_0)$ is
anticyclotomic, $d$ equals the
number of primes of $K$ over $p$ (if $e=p-1$) or zero (if $e<p-1$)). Let
$L'/F(\chi_0)$ be the maximal abelian pro-$p$ extension of $F(\chi_0)$
unramified outside $p$. (The group $\Gal(L/F(\chi_0))$ is the maximal
quotient of $\Gal(L'/F(\chi_0))$ of exponent $p$.) By Class Field Theory
$\rk_p(M/\ov{\mE})\leq \rk_p(\Gal(L'/F(\chi_0)))$
(equality holds if the $p$-part of the class group of $F(\chi_0)$ is
trivial).
We have $\Gal(L'/F(\chi_0)) \cong \bfZ_p^{r_2+1} \oplus \mT_p$, where
$\mT_p$
denotes the torsion subgroup.
So, if we show that $\rk_p(\mT_p)= d$, we are done.
This
follows immediately from
condition (4) of Definition \ref{adm} and \cite{Gras03},
Proposition 4.2.2 (p. 283). This completes the proof of the Lemma.
\end{proof}
We are now ready to complete the proof of Theorem \ref{essuni}.
Recall that the tensor product $\Gal(L_0/F(\chi_0))\otimes \ov{\bfF}_p$
is both a
quotient of $(M/T)\otimes \ov{\bfF}_p$ and of $(M/\ov{\mE})\otimes
\ov{\bfF}_p$. Since $(M/T) \otimes \ov{\bfF}_p =
(M_{\fp}/T_{\fp})\otimes \ov{\bfF}_p \times
(M_{\ov{\fp}}/T_{\ov{\fp}})\otimes \ov{\bfF}_p$, Lemma \ref{propdim}
implies that $$(M/T) \otimes \ov{\bfF}_p =\prod_{\psi \in G^{\vee}}
\left(\ov{\bfF}_p^{\psi}\times \ov{\bfF}_p^{\psi}\right).$$ On the other
hand one has
$$\ov{\mE}\otimes
\ov{\bfF}_p = \prod_{\psi
\in G^{\vee}\setminus \{1\}} \ov{\bfF}_p^{\psi}.$$ Using Lemma \ref{Gras2}
and the fact that $p$ annihilates $T$,
one can easily show the injectivity of the composite $$\ov{\mE}\otimes
\ov{\bfF}_p
\xrightarrow{\iota}
M\otimes \ov{\bfF}_p \xrightarrow{\pi} (M/T)\otimes \ov{\bfF}_p,$$ where
$\pi$ is the natural projection. So, $\Gal(L_0/F(\chi_0))\otimes
\ov{\bfF}_p$ is a quotient of $$((M/T)\otimes \ov{\bfF}_p)/(\ov{\mE}
\otimes
\ov{\bfF}_p) \cong \ov{\bfF}_p^1 \times \ov{\bfF}_p^1 \times \prod_{\psi
\in
G^{\vee}\setminus \{1\}}
\ov{\bfF}_p^{\psi}.$$ Since $\chi_0 \neq 1$, we have $\dim_{\ov{\bfF}_p}
(\Gal(L_0/F(\chi_0))\otimes
\ov{\bfF}_p) \leq 1$, which we wanted to show.\end{proof}

\begin{cor} \label{essuni2} Suppose $\rho': G_{\Sigma} \rightarrow
\GL_2(\bfF)$
is a
Galois representation satisfying conditions \textup{(Red)} and
\textup{(Sc)} for a
$\Sigma$-admissible character $\chi_0$. Then $\rho' \cong \rho_0$. 
\end{cor}

\section{Modular Forms and Galois representations}\label{Irreducible
modular deformations of rho_0}

In this section we exhibit irreducible ordinary Galois
representations that are residually reducible and arise from weight 2
cuspforms.

\subsection{Eisenstein congruences} \label{s4.1}
Let
$\phi_1,\phi_2$ be two Hecke characters with
infinity types $\phi_{1}^{(\infty)}(z)=z$
and $\phi_{2}^{(\infty)}(z)=z^{-1}$. Put $\gamma=\phi_1 \phi_2$. Write
$\fM$ for the conductor of $\phi:=\phi_1/\phi_2$.

Denote by $\mathfrak{S}$ the finite set of places where both
$\phi_i$ are ramified, but $\phi$ is unramified. Write
$\mathfrak{M}_i$ for the conductor of $\phi_i$. For an ideal
$\mathfrak{N}$ in $\mathcal{O}_F$ and a finite place ${\fq}$ of $F$
put $\mathfrak{N}_{\fq}=\mathfrak{N} \mathcal{O}_{F,{\fq}}$. We
define
$$K^1(\mathfrak{N}_{\fq})=\left \{
\begin{pmatrix}
  a & b \\
  c & d
\end{pmatrix}
\in {\rm GL}_2(\mathcal{O}_{F,{\fq}}), a-1, c \equiv 0 \,
\mod{\mathfrak{N}_{\fq}} \right \},$$ and
$$U^1(\mathfrak{N}_{\fq})=\{k \in {\rm GL}_2(\mathcal{O}_{F,{\fq}}):
{\rm det}(k) \equiv 1 \mod{\mathfrak{N}_{\fq}} \}.$$ Now put \be
\label{ourcompact} K_{\tuf}:=\prod_{{\fq} \in \fS}
U^1(\mathfrak{M}_{1,{\fq}}) \prod_{{\fq} \notin \fS}
K^1((\mathfrak{M}_1 \mathfrak{M}_2)_{\fq}) \subset
G(\bfA_{\tuf}).\ee

From now on, let $\Sigma$ be a finite set of places of $F$ containing
$$S_{\phi}:=\{{\fq} \mid \fM \fM^c \mathfrak{M}_1 \mathfrak{M}_1^c\} \cup \{{\fq}
\mid p d_F\}.$$


We denote by $\bfT(\Sigma)$ the $\Oo$-subalgebra of
$\End_{\Oo}(S_2(K_{\tuf},\gamma))$
generated by the Hecke operators
$T_{\fq}$
for all places $\fq \not\in \Sigma$. Following \cite{Taylorthesis}
(p. 107) we define idempotents $e_{\fp}$ and $e_{\ov{\fp}}$,
commuting with each other and with $\bfT(\Sigma)$ acting on
$S_2(K_{\tuf},\gamma)$. They are characterized by the property that
any element $h \in X:= e_{\fp} e_{\ov{\fp}} \hs
S_2(K_{\tuf},\gamma)$ which is an eigenvector for $T_{\fp}$ and
$T_{\ov{\fp}}$ satisfies $|a_{\fp}(h)|_p = |a_{\ov{\fp}}(h)|_p=1$,
where $a_{\fp}(h)$ (resp. $a_{\ov{\fp}}(h)$) is the
$T_{\fp}$-eigenvalue (resp. $T_{\ov{\fp}}$-eigenvalue) corresponding
to $h$. Let $\bfT^{\tuord}(\Sigma)$ denote the quotient algebra of
$\bfT(\Sigma)$ obtained by restricting the Hecke operators to $X$.

Let $J(\Sigma) \subset \bfT(\Sigma)$ be the ideal generated by
$$\{T_{\fq} - \phi_1(\varpi_{\fq}) \cdot \#(\Oo_{F, \fq}/\varpi_{\fq})-
\phi_2(\varpi_{\fq}) \mid \fq \not\in \Sigma \}.$$

\begin{definition} Denote by $\fm(\Sigma)$ the
maximal ideal of $\bfT^{\tuord}(\Sigma)$ containing the image of
$J(\Sigma)$. We set $\bfT_{\Sigma}:=
\bfT^{\tuord}(\Sigma)_{\fm(\Sigma)}$. Moreover, set $J_{\Sigma}:=
J(\Sigma) \bfT_{\Sigma}$. We refer to $J_{\Sigma}$ as the
\emph{Eisenstein ideal of $\bfT_{\Sigma}$}.
\end{definition}

\begin{thm} [\cite{Berger05}, Theorem 6.3, \cite{B09}
Theorem 14] \label{Eiscong} Let $\phi$ be an unramified Hecke
character of infinity type $\phi^{(\infty)}(z)=z^2$. There exist
Hecke characters $\phi_1, \phi_2$ with $\phi_1/\phi_2=\phi$ such
that their conductors are divisible only by ramified primes or inert
primes not congruent to $\pm 1 \mod{p}$ and such that $$\#
(\bfT_{\Sigma}/J_{\Sigma}) \geq \#(\Oo/(L^{\rm int}(0,\phi))).$$
\end{thm}

\begin{proof}
The Eisenstein cohomology class used in the proof of
\cite{B09} Theorem 14 is ordinary, so we can deduce the
statement for the ordinary cuspidal Hecke algebra.
\end{proof}

\begin{rem}
If $\phi$ is unramified then $\overline{\phi_{\mathfrak{p}}
\epsilon}$ is anticyclotomic (see \cite{B09} Lemma 1).
The condition on the conductor of the auxiliary character $\phi_1$
together with our assumption on the discriminant of $F$ therefore
ensure that for $\chi_0=\overline{\phi_{\mathfrak{p}} \epsilon}$
condition (2) of $\Sigma$-admissibility is automatically satisfied
for all primes $\fq \in S_{\phi}$.
\end{rem}

The assumption on the ramification of $\phi$ can be relaxed. For
example, Proposition 16 and Theorem 28 of \cite{Berger08} and
Proposition 9 and Lemma 11 of \cite{B09} imply the
following:

\begin{thm} \label{Eiscong+}
Let $\phi_1$, $\phi_2$ be as at the start of this section. Assume both
$\fM_1$
and $\fM_2$ are coprime to
$(p)$ and divisible
only by primes split in $F/\bfQ$ and that $p \nmid
\#(\mathcal{O}_F/\fM \fM_1)^{\times}$. Suppose
$(\phi_1/\phi_2)^c=\overline{\phi_1/\phi_2}$. If
the torsion part of $H^2_c(S_{K_{\tuf}},\bfZ_p)$ is trivial, where
$$S_{K_{\tuf}}=G(\bfQ) \backslash G(\bfA)/K_{\tuf} U(2) \bfC^{\times}$$
then $$\# (\bfT_{\Sigma}/J_{\Sigma}) \geq \#(\Oo/(L^{\rm
int}(0,\phi_1/\phi_2))).$$
\end{thm}

\begin{rem}\label{r4.5}
In fact, by replacing $\bfZ_p$ by the appropriate coefficient
system, the result is true for characters $\phi_1 ,\phi_2$ of
infinity type $z \ov z^{-m}$ and $z^{-m-1}$, respectively, for $m
\geq 0$. For Theorems \ref{Eiscong} and \ref{Eiscong+}, the field E needs 
to 
contain the
values of the finite parts of $\phi_1$ and $\phi_2$ as well as 
$L^{\tuint}(0,\phi_1/\phi_2)$.
\end{rem}

We will from now on assume that we are either in the situation of
Theorem \ref{Eiscong} or \ref{Eiscong+} and fix the characters
$\phi_1, \phi_2$ and $\phi=\phi_1/\phi_2$, with corresponding
conditions on the set $\Sigma$ and definitions of $K_{\tuf}$,
$\bfT_\Sigma$, and $J_{\Sigma}$. We also assume from now on that
$\val_p(L^{\rm int}(0,\phi))>0$. Put
$\chi_0=\overline{\phi_{\mathfrak{p}} \epsilon}$ and assume that
$\chi_0$ is $\Sigma$-admissible. If we are in the situation of
Theorem \ref{Eiscong+} then suppose also that $\fM_1$ and $\fM_2$
are not divisible by any primes $\fq$ such that $\#(\Oo/\fq) \equiv
1 \mod{p}$. (This last assumption is only used  in the proof of
Theorem~\ref{minimal2}.)

\subsection{Residually reducible Galois representations} \label{defining
rho_0}

Write
$$S_2(K_{\tuf}, \gamma)_{\fm(\Sigma)}=\bigoplus_{\pi \in \Pi_\Sigma}
\pi_{\tuf}^{K_{\tuf}}$$ for a finite set $\Pi_\Sigma$ of ordinary
cuspidal automorphic representations with central character
$\gamma$, such that $\pi_{\tuf}^{K_{\tuf}} \neq 0$. The set
$\Pi_\Sigma$ is non-empty by Theorem \ref{Eiscong} under our
assumption that $\val_p(L^{\rm int}(0,\phi))>0$.

Let $\pi\in \Pi_{\Sigma}.$ Let $\rho_{\pi}: G_{\Sigma} \rightarrow
\GL_2(E)$ be the Galois representation attached to $\pi$ by Theorem
\ref{attach1} (This is another point where we assume that $E$ is large 
enough). The condition on the central character in Theorem
\ref{attach1} can be satisfied  (after possibly twisting with a
finite character) under our assumptions on $\phi$, see
\cite{B09} Lemma 8. The representation $\rho_{\pi}$ is
unramified at all $\fq \notin
S_{\phi}$, and satisfies
$$\tr
\rho_{\pi}(\Frob_{\fq}) = a_{\fq}(\pi)$$ and $$\det
\rho_{\pi}(\Frob_{\fq}) = \gamma(\varpi_{\fq}) \cdot \#(\OF/\fq).$$

By definition, $\bfT_{\Sigma}$ injects into $\bigoplus_{\pi \in
\Pi_\Sigma} \End_{\Oo}(\pi^{K_{\tuf}})$. Since $T_{\fq}$ acts on
$\pi$ by multiplication by $a_{\fq}(\pi)~\in~\mathcal{O}$ the Hecke
algebra
$\bfT_{\Sigma}$ embeds, in fact, into $B=\prod_{\pi \in \Pi_\Sigma}
\mathcal{O}$.

Observe that $\bigoplus_{\pi \in \Pi_\Sigma} \tr \rho_{\pi}(\sigma)
\in \bfT_{\Sigma}\subset B$ for all $\sigma \in G_{\Sigma}$. This
follows from the Chebotarev Density Theorem and the continuity of
$\rho_{\pi}$ (note that $\bfT_{\Sigma}$ is a finite $\Oo$-algebra).

Fix $\pi \in \Pi_\Sigma$ for the rest of this subsection. Define
$\rho'_{\pi}:= \rho_{\pi} \otimes \phi_{2,\fp}^{-1}$.
Then $\rho'_{\pi}$
satisfies
$$\tr \rho'_{\pi} (\Frob{\fq}) \equiv 1+ (\phi_{\fp}
\epsilon)(\Frob{\fq}) \pmod{\varpi} \quad \textup{for} \hs \fq\notin
S_{\phi},$$ and $$\det \rho'_{\pi} = \gamma \cdot \phi_{2, \fp}^{-2}
\cdot \epsilon = \phi_{\fp} \epsilon.$$ By choosing a suitable
lattice $\Lambda$ one can ensure that $\rho'_{\pi}$ has image inside
$\GL_2(\mathcal{O})$. The Chebotarev Density Theorem and the
Brauer-Nesbitt Theorem imply that
$$(\ov{\rho}'_{\pi})^{\tuss} \cong 1 \oplus \ov{\phi}_{\fp}
\ov{\epsilon}.$$

By Theorem \ref{attach1} $\rho'_{\pi}$ is irreducible, so a standard
argument (see e.g. Proposition 2.1 in \cite{Ribet76}) shows the lattice
$\Lambda$ may be chosen in such a way
that $\ov{\rho}'_{\pi}$ is not semi-simple and \be \label{choicela}
\ov{\rho}'_{\pi} = \bmat 1 & * \\ & \ov{\phi}_{\fp}
\ov{\epsilon}\emat.\ee
Hence
$\ov{\rho}'_{\pi}$ satisfies conditions (Red) and (Sc) of Section
\ref{Uniqueness of a certain residual Galois representation}. By
Theorem \ref{Urbanordinary}, $\rho'_{\pi}$
is ordinary which combined with (\ref{choicela}) implies
that \be \label{split1}\ov{\rho}'_{\pi}|_{D_{\ov{\fp}}} \cong \bmat 1 \\
& (\ov{\phi}_{\fp} \ov{\epsilon})|_{D_{\ov{\fp}}}\emat.\ee We put
\begin{equation} \label{defrho}
\rho_0:=\ov{\rho}'_{\pi}.\end{equation}

\begin{rem} \label{specialbasis} Let $\tau \in I_{\fp}$ be as in
Section \ref{Uniqueness of a certain residual Galois representation}. The
isomorphism in (\ref{split1})
implies that one can find a basis such that
$$\rho_0(\tau) =
\bmat 1
\\ & \chi_0(\tau) \emat,$$ and
$$\rho_0(g_0) = \bmat 1&1 \\ & 1\emat$$ for a fixed $g_0 \in
I_{\fp}$. Note that such a $g_0$ exists as it follows from the Proof
of Theorem \ref{essuni} that for $\rho_0$ satisfying (\ref{split1}),
the extension $F(\rho_0)/F(\chi_0)$ is totally ramified at $\fp$.

Furthermore, the ordinary modular deformations of $\rho_0$ in
Section 5.2 cannot be induced from a character of a quadratic
extension of $F$ because such representations split when restricted
to the decomposition groups $D_{\fq}$ for ${\fq} \mid p$. This
follows from Urban's result (Theorem 2.3) and the restriction of
these characteristic 0 representations being semisimple on an open
subgroup of each of the decomposition groups.

\end{rem}

\section{Deformations of $\rho_0$}
Let $\Sigma$, $\phi$, $\chi_0$ and $\rho_0$ be as in Section
\ref{Irreducible
modular deformations of rho_0}. Recall that we have assumed that
$\chi_0$ is $\Sigma$-admissible and have shown in Section
\ref{defining
rho_0} that
$\rho_0$ satisfies conditions (Red) and (Sc) of Section
\ref{Uniqueness of a certain residual Galois representation}.
Hence by Corollary \ref{essuni2}, $\rho_0$
is unique up to isomorphism. By (\ref{split1}) the extension
$F(\rho_0)/F(\chi_0)$ is ramified at $\fp$ but splits at
$\ov{\fp}$. In this section we study
deformations of $\rho_0$.

\subsection{Definitions} \label{Deformations of rho_0}
Denote the category of local complete Noetherian $\Oo$-algebras with
residue field $\bfF$ by $\textup{LCN}(E)$. An $\Oo$-deformation of
$\rho_0$ is a pair consisting of $A \in \textup{LCN}(E)$ and an
equivalence class of continuous representations $\rho: G_{\Sigma}
\rightarrow \GL_2(A)$ such that $\rho_0 = \rho \pmod{\fm_A}$, where
$\fm_A$ is the maximal ideal of $A$. As is customary we will denote
a deformation by a single member of its equivalence class.  Note
that the Hodge-Tate weights of $\phi_{\mathfrak{p}} \epsilon$ are -1
at $\mathfrak{p}$ and +1 at $\ov {\mathfrak{p}}$.


\begin{definition} \label{ordinary} We say that an $\Oo$-deformation
$\rho: G_{\Sigma} \rightarrow \GL_2(A)$ of $\rho_0$ is
\emph{ordinary} if
$$\det \rho = \phi_{\fp} \epsilon$$ and
$$\rho|_{D_{\fp}} \cong \bmat \Psi_1 & * \\ & \Psi_2\emat $$ with
$\Psi_1$ unramified and $$\rho|_{D_{\ov{\fp}}} \cong \bmat \Psi_3 &
* \\ & \Psi_4\emat $$ with $\Psi_4$ unramified.
\end{definition}

Following \cite{SkinnerWiles97} we
make the following definition:

\begin{definition} \label{sigmamin} We will say that
a deformation $\rho$ of $\rho_0$ is \emph{$\Sigma$-minimal} if
$\rho$ is ordinary and for all primes $\fq\in \Sigma$ such that $\#
(\OF/\fq) \equiv 1$ \textup{(mod~$p$)} one has $$\rho|_{I_{\fq}} \cong 
\bmat 1
\\ &
\phi_{\fp}|_{I_{\fq}} \emat.$$
Note that by our assumption on the conductor of
$\phi$, we in fact have $\phi_{\fp}|_{I_{\fq}} = 1$
for $\fq$ as above. \end{definition}

Since $\rho_0$ has a scalar centralizer and $\Sigma$-minimality is a
deformation condition in the sense of \cite{Mazur97}, there exists a
universal deformation ring which we will denote by $R_{\Sigma,
\Oo} \in \textup{LCN}(E)$, and a universal
$\Sigma$-minimal $\Oo$-deformation $\rho_{\Sigma, \Oo} :
G_{\Sigma} \rightarrow \GL_2(R_{\Sigma, \Oo})$ such that
for every $A \in \textup{LCN}(E)$ there is a one-to-one
correspondence between the set of $\Oo$-algebra maps $R_{\Sigma,
\Oo} \rightarrow A$ (inducing identity on $\bfF$) and the
set of $\Sigma$-minimal deformations $\rho: G_{\Sigma} \rightarrow
\GL_2(A)$ of $\rho_0$.







\subsection{Irreducible modular deformations of $\rho_0$} The arguments
from Section \ref{defining rho_0} together with the
uniqueness of $\rho_0$ (Corollary \ref{essuni2}) can now be
reinterpreted as:
\begin{thm} \label{minimal2} For any $\pi \in \Pi_\Sigma$ there is an
$\Oo$-algebra homomorphism $r_{\pi}: R_{\Sigma, \Oo}
\twoheadrightarrow \mathcal{O}$ inducing $\rho'_{\pi}$.\end{thm}

\begin{proof}
The only property left to be checked is $\Sigma$-minimality. This is
clear since $\rho_{\pi}$ is unramified away from $S_{\phi}$, and no
$\fq \in S_{\phi}$ satisfies $\# (\OF/\fq) \equiv 1\pmod{p}$ by
construction (if we are in the case of Theorem \ref{Eiscong}) or
assumption (in the case of Theorem \ref{Eiscong+}).
\end{proof}

\begin{rem}
The assumption on the conductors of $\phi_1, \phi_2$ made at the end
of Section~\ref{s4.1} could be relaxed if local-global compatibility
was known for the Galois representations constructed by Taylor. For
a discussion of the Langlands conjecture for imaginary quadratic
fields see \cite{CalegariDunfield06} Conjecture 3.2.
\end{rem}

\begin{prop} \label{nored} There does not exist any non-trivial
upper-triangular $\Sigma$-minimal deformation of $\rho_0$ to
$\GL_2(\bfF[x]/x^2)$.
 \end{prop}

\begin{proof} Let $\rho: G_\Sigma \rightarrow \GL_2(\bfF[x]/x^2)$ be an
upper-triangular $\Sigma$-minimal deformation. Then $\rho$ has the form
$$\bmat 1+ x
\alpha & *
\\ & \chi_0 + x\beta \emat $$ for $\alpha: G_\Sigma \rightarrow \bfF^+$ a
group homomorphism (here $\bfF^+$ denotes the additive group of
$\bfF$) and $\beta: G_\Sigma \rightarrow \bfF$ a function.

By ordinarity of $\rho$ we have $\det \rho = \chi_0$, which forces
$\beta = - \a \chi_0$. Let $\fq$ be a prime of $F$ and consider the
restriction of $\alpha$ to $I_{\fq}$. If $\fq \in \Sigma$, $\fq \neq
\fp, \ov{\fp}$ and $\#(\OF/ \fq) \not\equiv 1$ mod $p$, one must have (by
local class field theory) that $\a (I_{\fq})=0$. If $\fq \in \Sigma$
and $\#(\OF/ \fq) \equiv 1$ mod $p$ (resp. $\fq=\fp$), then
$\Sigma$-minimality (resp. ordinarity at $\fp$) implies that
$\a(I_{\fq})=0$. Thus $\alpha$ can only be ramified at $\ov{\fp}$.
However, since $\rho$ is ordinary at $\ov{\fp}$, $\rho|_{I_{\ov{\fp}}}$
can be conjugated to a representation of the form $$\bmat 1 \\ *& \chi_0
\emat.$$ This, together with the fact that $\chi_0$ is ramified at
$\ov{\fp}$ (see the remark after Definition \ref{adm}) easily implies
that $\alpha$ must be unramified at $\ov{\fp}$.
Since $p \nmid \# \Cl_F$, we must have
$\alpha=0$. Hence $\rho$ is of the form $$\bmat 1 & * \\ &
\chi_0\emat$$ and for $G'={\rm ker}(\chi_0) \subset G_\Sigma$ we
have
$$\rho|_{G'} = \bmat 1 & b_0 + x b_1 \\ & 1 \emat$$ for $b_0, b_1:
G' \rightarrow \bfF^+$ group homomorphisms. Note that
$F(\rho)/F(\chi_0)$ is thus an abelian extension unramified outside
$\Sigma$ which is anihilated by $p$. Moreover, $\Gal(F(\chi_0)/F)$
acts on $\Gal(F(\rho)/F(\chi_0))$ via $\chi_0^{-1}$. Since $\chi_0$
is $\Sigma$-admissible, Theorem \ref{essuni} implies that
$\Gal(F(\rho)/F(\chi_0)) \cong \bfF$. Since $\rho \equiv \rho_0$ mod
$x$, we see that $b_1=0$ and thus $\rho$ must be the trivial
deformation of $\rho_0$.
\end{proof}

\begin{prop}
The universal deformation ring $R_{\Sigma, \Oo}$ is
generated as an $\Oo$-algebra by traces.
\end{prop}

\begin{proof}
We follow the argument of \cite{Calegari06}, Lemma 4.2. If suffices
to show that any non-trivial deformation of $\rho_0$ to ${\rm
GL}_2(\bfF[x]/x^2)$ is generated by traces. Let $\rho$ be such a
deformation. Observe that for $\sigma \in {\rm Gal}(\ov
\bfQ/F(\chi_0))$ the element $\rho(\sigma)$ can be written as
$$\begin{pmatrix}1+xa(\sigma)&b_0(\sigma)+
xb_1(\sigma)\\xc(\sigma)&1+x d(\sigma)\end{pmatrix},$$ so ${\rm
det}(\rho)(\sigma)-{\rm tr}(\rho)(\sigma)=-1-xb_0(\sigma)
c(\sigma)$. Since $c$ is non-trivial by Proposition
\ref{nored}, the Chebotarev Density Theorem implies
there exists a $\sigma$ such that $xb_0(\sigma) c(\sigma)\neq 0$.
Since ${\rm det}(\rho)(\sigma)=1$, it follows that the traces of
$\rho$ generate $\bfF[x]/x^2$.
\end{proof}

\begin{lemma} \label{imaget} The image of the map
$R_{\Sigma, \Oo} \rightarrow \prod_{\pi \in
\Pi_{\Sigma}}\mathcal{O}$ given by $x \mapsto
(r_{\pi}(x))_{\pi}$ is $\bfT_{\Sigma}$.
\end{lemma}

\begin{proof} The $\Oo$-algebra $R_{\Sigma, \Oo}$ is generated by
the set $\{ \tr \rho_{\Sigma, \Oo}(\Frob_{\fq}) \mid \fq \not \in
\Sigma\}$. For $\fq \not \in \Sigma$, we have $$r_{\pi}(\tr
\rho_{\Sigma, \Oo}(\Frob_{\fq})) = \phi_{2,\fp}(\Frob_{\fq})^{-1}
a_{\fq}(\pi).$$ Hence the image of the map in the Lemma is the
closure of the $\Oo$-subalgebra of $\prod_{\pi \in \Pi_{\Sigma}}
\mathcal{O}$ generated by the set
$\{\phi_{2,\fp}(\Frob_{\fq})^{-1} T_{\fq} \mid \fq \not\in \Sigma
\}$ which is the same as the closure of the $\Oo$-subalgebra of
$\prod_{\pi \in \Pi_{\Sigma}} \mathcal{O}$ generated by the set
$\{T_{\fq} \mid \fq \not\in \Sigma \}$ which in turn is
$\bfT_{\Sigma}$.
 \end{proof}

By Lemma \ref{imaget} we obtain a surjective $\Oo$-algebra
homomorphism $r: R_{\Sigma, \Oo} \twoheadrightarrow
\bfT_{\Sigma}$.

\begin{thm} \label{heckedef} If $R_{\Sigma,\Oo}$ is a discrete
valuation ring and if $${\rm val}_p(L^{\rm int}(0,\phi))>0$$ then the
map $r: R_{\Sigma, \Oo} \rightarrow \bfT_{\Sigma}$ defined
above is an isomorphism.
\end{thm}

\begin{proof}
As in \cite{Calegari06} this follows because $\bfT_\Sigma \neq
\bfT_\Sigma/\varpi^n$ for any $n$.
\end{proof}

\subsection{When is $R_{\Sigma, \Oo}$ a dvr?}\label{when is}

Set $\Psi:=\phi_{\fp}\epsilon$ and write
$\Psi_2$ for
$\Psi$ \textup{(mod $\varpi^2$)}.

\begin{prop} \label{prop5.9}
Assume that $\rho_0$ does not admit any
$\Sigma$-minimal upper-triangular deformation to $\GL_2(\Oo/\varpi^2
\Oo)$ and that $\chi_0^{-1}$ is $\Sigma$-admissible. Then $R_{\Sigma,
\Oo}$ is a discrete valuation ring. \end{prop}

\begin{rem} \label{remcond} The condition on the non-existence of a
$\Sigma$-minimal
upper-triangular deformation of $\rho_0$ to $\GL_2(\Oo/\varpi^2 \Oo)$
follows from the following condition on the character $\phi$ (or, which is
the
same, on the splitting field $F(\Psi_2)$ of $\Psi_2$): There does not
exist an abelian
$p$-extension $L$ of $F(\Psi_2)$,
unramified
outside $\fp$
such that
$\Gal(L/F (\Psi_2))$ is isomorphic to a $\bfZ[\Gal(F(\Psi_2)/F)]$-submodule of
$(\mathcal{O}/\varpi^2\mathcal{O})(\Psi_2^{-1})$ on which
$\Gal(F(\Psi_2)/F)$ operates faithfully. Indeed, as in
the proof of Proposition \ref{nored}, the condition of $\Sigma$-minimality
forces any such
deformation to be of the form $\bmat 1 & * \\ 0 & \Psi_2\emat $ with
$*$ corresponding to an extension of $F(\Psi_2)$ unramified away from
$\fp$.  \end{rem}



\begin{proof}[Proof of Proposition \ref{prop5.9}]
We briefly recall some general facts about Eisenstein
representations from Section 3 of \cite{Calegari06} and Section 2 of
\cite{BellaicheChenevier06}: Let $(A,\fm,k)$ be a local $p$-adically
complete ring. Let $G$ be a topological group and consider a
continuous representation $\rho:G \to {\rm GL}_2(A)$ such that
$\tr (\rho)$ mod $\fm$ is the sum of two distinct characters $\tau_i:G
\to k^{\times}, i=1,2$.

\begin{definition}
The \emph{ideal of reducibility} of $A$ is the smallest ideal $I$
of $A$ such that $\tr(\rho)$ mod $I$ is the sum of two
characters.
\end{definition}

\begin{lemma}[\cite{BellaicheChenevier06} Corollaire 2,
\cite{Calegari06} Lemma 3.4] \label{dvr1} Suppose $A$ is noetherian, that
the
ideal of reducibility is maximal, and that $${\rm dim}_k {\rm
Ext}^1_{{\rm cts}, k[G]}(\tau_1,\tau_2)={\rm dim}_k {\rm
Ext}^1_{{\rm cts}, k[G]}(\tau_2,\tau_1)=1.$$ If $A$ admits a
surjective map to a ring of characteristic $0$, then  $A$ is a
discrete valuation ring.
\end{lemma}

We apply this Lemma for $G=G_{\Sigma}$, $A=R_{\Sigma, \Oo}$, $\tau_1=1$,
and $\tau_2=\chi_0$. $\Sigma$-admissibility of both $\chi_0$ and its
inverse guarantees that the dimension condition in Lemma \ref{dvr1} is
satisfied.
Moreover, since $R_{\Sigma, \Oo} \rightarrow \bfT_{\Sigma}$ is surjective
and
$\bfT_{\Sigma}$ is a ring of characteristic zero,
we infer that $R_{\Sigma, \Oo}$ is a
discrete valuation ring whenever the ideal of reducibility $I$ of
$R_{\Sigma, \Oo}$ is maximal. This is the case if and only if there
does not exist a surjection $R_{\Sigma, \Oo}/I \twoheadrightarrow
\bfF[x]/x^2$ or $R_{\Sigma, \Oo}/I \twoheadrightarrow \Oo/\varpi^2
\Oo$, or, by the universality of $R_{\Sigma, \Oo}$ if
$\rho_0$ does not admit any non-trivial $\Sigma$-minimal
deformations of $\rho_0$ to ${\rm GL}_2(\bfF[x]/x^2)$ or ${\rm
GL}_2(\Oo/\varpi^2 \Oo)$ that are upper-triangular.
The latter cannot occur by assumption and the former by
Proposition \ref{nored}.
\end{proof}


Note that $\Gal(F(\Psi)/F) \cong \Gamma \times \Delta$ with $\Gamma \cong
\bfZ_p$ and $\Delta$ a finite group. 


\begin{thm} \label{prop5.10} 
Assume 
$p \nmid \# \Delta$ and that $\chi_0^{-1}$ is
$\Sigma$-admissible. If $$\# 
(\Oo/L^{\tuint}(0,\phi))=p^{[\Oo:\bfZ_p]},$$
then $\rho_0$ does
not admit any
$\Sigma$-minimal upper-triangular deformation to $\GL_2(\Oo/\varpi^2
\Oo)$. In particular $R_{\Sigma,
\Oo}$ is a discrete valuation ring. \end{thm}

\begin{rem} \label{OtoO'} Let $\Oo'$ be the ring of integers
in
any finite extension of $\bfQ_p$ containing 
$L^{\tuint}(0,\phi)$.
Note 
that 
the $L$-value condition in Theorem \ref{prop5.10} is 
equivalent to $\#(\Oo'/L^{\tuint}(0,\phi))=p^{[\Oo':\bfZ_p]}$. 
\end{rem}

\begin{proof} 
Write $X_{\infty}$ for $\Gal(M(F(\Psi))/F(\Psi))$ with
$M(F(\Psi))$ the
maximal abelian pro-$p$-extension of $F(\Psi)$ unramified away from
the primes lying over $\fp$ and $(X_{\infty}\otimes \Oo)^{\chi_0^{-1}}$
the
$\chi_0^{-1}$-part of $X_{\infty}\otimes \Oo$. Moreover, write
$M(F(\Psi_2))_{\Psi}$ for the maximal abelian pro-$p$-extension of
$F(\Psi_2)$
unramified away
from $\fp$ on which $\Gal(F(\Psi_2)/F)$ acts via $\Psi^{-1}$.
We will use the following two lemmas.

\begin{lemma} \label{lem1}
We have $$\#
((X_{\infty}\otimes
\Oo)^{\chi_0^{-1}}/(\gamma-\Psi^{-1}(\gamma))(X_{\infty}\otimes
\Oo)^{\chi_0^{-1}})
\leq \#(\mathcal{O}/L^{\rm int}(0,\phi)).$$
\end{lemma}

\begin{lemma} \label{lem2} We have
\begin{multline} \#(\Gal(M(F(\Psi_2))_{\Psi}/F(\Psi_2))\otimes
\Oo)^{\chi_0^{-1}}
\leq\\
\leq \# ((X_{\infty}\otimes
\Oo)^{\chi_0^{-1}}/(\gamma-\Psi^{-1}(\gamma))(X_{\infty}\otimes
\Oo)^{\chi_0^{-1}}).\end{multline}
\end{lemma}


We first show how Theorem \ref{prop5.10} follows from
these lemmas.
Suppose that $L$ as in Remark \ref{remcond} existed. Then one would
have
$L
\subset M(F(\Psi_2))_{\Psi}$. One also
has
$F(\Psi_2)F(\rho_0) \subset M(F(\Psi_2))_{\Psi}$, hence \begin{multline} 
\label{ine} \#
(\Gal(F(\Psi_2)F(\rho_0)/F(\Psi_2))\otimes
\Oo)^{\chi_0^{-1}} \leq \\
\leq \#(\Gal(M(F(\Psi_2))_{\Psi}/F(\Psi_2))\otimes
\Oo)^{\chi_0^{-1}},\end{multline}
 but $F(\Psi_2)F(\rho_0) \neq L$, as $\Gal(F(\Psi_2)/F)$ does not
act faithfully on the group $\Gal(F(\Psi_2)F(\rho_0)/F(\Psi_2))$. It
is easy to see that the quantity on the left-hand side of
(\ref{ine}) is $p^{[\Oo:\bfZ_p]}$. 
Hence, if the conditions of Theorem
\ref{prop5.10} are satisfied, the inequalities in Lemmas \ref{lem1}
and \ref{lem2} become equalities and this easily implies that
$F(\Psi_2)F(\rho_0)= M(F(\Psi_2))_{\Psi}$. Thus $L$ cannot exist.
\end{proof}

\begin{proof}[Proof of Lemma \ref{lem1}]
It follows from
Proposition
\ref{torsion1} (see section \ref{A
reducible
deformation of rho_0})
that the module ${\rm
Hom}(X_{\infty},(E/\mathcal{O})(\Psi^{-1}))^{{\rm Gal}(F(\Psi)/F)}$ is finite.
 For any Galois character $\tau:G_F \to \Oo^{\times}$ put
$A_{\tau}=E/\mathcal{O}(\tau)$. By \cite{Guo93a}
Proposition 2.2(i) and Proposition 2.3 $${\rm
Hom}(X_{\infty},(E/\mathcal{O})(\Psi^{-1}))^{{\rm Gal}(F(\Psi)/F)}
\cong {\rm S}^{\rm str}_{A_{\Psi^{-1}}}(F),$$ where ${\rm S}^{\rm
str}_{A_{\Psi^{-1}}}(F) \subset H^1({\rm G_F},A_{\Psi^{-1}})$
denotes the strict Selmer group defined by Greenberg (see \cite{Guo93a}, 
section 1 for definition). Note that the
class number restriction in \cite{Guo93a} is not required for these
results.

It is clear that ${\rm S}^{\rm str}_{A_{\Psi^{-1}}}(F)\cong {\rm
S}^{\rm str}_{A_{(\Psi^{-1})^c}}(F)={\rm S}^{\rm
str}_{A_{\Psi}}(F)$. The duality result of \cite{Guo93b} Theorem 2
implies an isomorphism $${\rm S}^{\rm str}_{A_{\Psi}}(F) \cong {\rm
S}^{\rm str}_{A_{\Psi^{-1} \epsilon}}(F)$$ if both Selmer groups are
finite. By the observation at the beginning of the proof we know
that ${\rm S}^{\rm str}_{A_{\Psi}}(F)$ is finite. For the Selmer
group of the dual character the arguments of the proof of
Proposition 2.2 of \cite{Guo93a} imply that
$${\rm S}^{\rm str}_{A_{\Psi^{-1} \epsilon}}(F) \hookrightarrow {\rm
Hom}(X_{\infty},(E/\mathcal{O})(\Psi^{-1} \epsilon))^{{\rm
Gal}(F(\Psi)/F)}.$$

By applying the Main Conjecture of Iwasawa theory Wiles
\cite{Wiles95} p. 532 proves that
$$\# {\rm Hom}(X_{\infty},(E/\mathcal{O})(\Psi^{-1} \epsilon))^{{\rm
Gal}(F(\Psi)/F)} \leq \#(\mathcal{O}/L^{\rm int}(0,\phi)).$$ (For
similar results towards the Bloch-Kato conjecture see also
\cite{Guo93a} who treats imaginary quadratic fields of class number
one but Hecke characters of general infinity types.) Finally, it is easy
to see that \begin{multline}\#{\rm
Hom}(X_{\infty},(E/\mathcal{O})(\Psi^{-1}))^{{\rm
Gal}(F(\Psi)/F)}=\\
=\#(X_{\infty}\otimes
\Oo)^{\chi_0^{-1}}/(\gamma-\Psi^{-1}(\gamma))(X_{\infty}\otimes
\Oo)^{\chi_0^{-1}}.\end{multline}
\end{proof}

\begin{proof}[Proof of Lemma \ref{lem2}] The restriction
provides a surjective
$\Oo$-linear homomorphism
$$(X_{\infty}\otimes \Oo)^{\chi_0^{-1}} \twoheadrightarrow
(\Gal(M(F(\Psi_2))/F(\Psi_2))\otimes \Oo)^{\chi_0^{-1}}.$$ Since
$\Gal(F(\Psi_2)/F)$ acts on $\Gal(M(F(\Psi_2))_{\Psi}/F(\Psi_2))$
via $\Psi^{-1}$ the composite \begin{multline}(X_{\infty}\otimes
\Oo)^{\chi_0^{-1}} \twoheadrightarrow
(\Gal(M(F(\Psi_2))/F(\Psi_2)\otimes \Oo)^{\chi_0^{-1}}
\twoheadrightarrow \\
\twoheadrightarrow (\Gal(M(F(\Psi_2))_{\Psi}/F(\Psi_2)) \otimes
\Oo)^{\chi_0^{-1}}\end{multline} clearly factors through
$$(X_{\infty}\otimes
\Oo)^{\chi_0^{-1}}/(\gamma-\Psi^{-1}(\gamma))(X_{\infty}\otimes
\Oo)^{\chi_0^{-1}}.$$ \end{proof}

\subsection{Modularity theorem}

In this section we state a modularity theorem which is a consequence of
the results of the previous sections.
To
make
its statement self-contained, we explicitly include all the assumptions
we have made
so far.

\begin{thm} \label{cor5.8} Let
$\phi_1$, $\phi_2$ be Hecke characters of $F$ with split conductors and of
infinity type $z$ and $z^{-1}$ respectively such that
$\phi:=\phi_1/\phi_2$ is unramified. Assume that the conductor
$\mathfrak{M}_1$ of $\phi_1$ is coprime to $(p)$ and that $p
\nmid \#(\Oo_F/\fM_1)^{\times}$. Moreover, assume that
$\val_p(L^{\tuint}(0, \phi))>0$.

Let $\rho: G_{\Sigma} \rightarrow \GL_2(E)$ be a continuous
irreducible representation that is ordinary at all places $\fq \mid
p$ (in the sense of Theorem \ref{Urbanordinary}). Suppose
$\ov{\rho}^{\tuss} \cong \chi_1
\oplus \chi_2$ with
$\chi_1=\overline{\phi_{1,\fp} \epsilon}$,
$\chi_2=\overline{\phi_{2,\fp}}$. Set $\chi_0: = \chi_1
\chi_2^{-1}$. If all of the following conditions are
satisfied:

\begin{enumerate}
  \item $\Sigma \supset \{{\fq} \mid p d_F \mathfrak{M}_1 \mathfrak{M}_1^c
  \}$,
  \item the representation $\ov{\rho}\otimes \chi_2^{-1}$ admits no
upper-triangular $\Sigma$-minimal deformation to $\GL_2(\Oo/\varpi^2
\Oo)$,
  \item $\chi_0$ and $\chi_0^{-1}$ are $\Sigma$-admissible
  \item ${\rm det}(\rho)=\phi_1 \phi_2 \epsilon$,
  \item $\rho \otimes \phi_{2,\fp}^{-1}$ is $\Sigma$-minimal,
\end{enumerate}
then $\rho$ is modular in the sense of Definition
\ref{modular1}.
\end{thm}

\begin{rem} \label{better} Write ${\rm Gal}(F(\Psi)/F)=\Gamma \times 
\Delta$ with $\Gamma \cong \bfZ_p$.
If $p \nmid \# \Delta$
then by Theorem \ref{prop5.10}
condition (2) in Theorem \ref{cor5.8} can be replaced by $\#
(\Oo/L^{\tuint}(0,\phi)) = p^{[\Oo:\bfZ_p]}.$ \end{rem}

\begin{rem}
Theorem~\ref{Eiscong+} and Remark~\ref{r4.5} show that the
conditions for the conductor and infinity type of $\phi$ can be
relaxed if one imposes a condition on the torsion-freeness of a
cohomology group.
\end{rem}

\begin{example} \label{example1}
We now turn to a numerical example in which we can verify the conditions
of Theorem
\ref{cor5.8} (under an additional assumption which we discuss below).
Let $F=\bfQ(\sqrt{-51})$ and $p=5$ (which splits in $F$). Since the
class number is $2$, there are two unramified Hecke characters of
infinity type $z^2$.
For each of them the functional equation relates the $L$-value at 0 to the
$L$-value at 0 of a Hecke character of infinity type $\ov{z}/z$. The
latter one in turn is equal (by the Weil lifting - see e.g.
\cite{Miyake89}, Theorem 4.8.2 or \cite{Iwaniec97}, Theorem
12.5) to the $L$-value at 1 of a weight 3 modular form of level 51 and
character the Kronecker symbol $\left(\frac{-51}{\cdot}\right)$.
Let $\phi$ be the Hecke character of infinity type $z^2$
corresponding to the modular form with $q$-expansion starting with
$q+3q^3+ \ldots$. Using MAGMA \cite{MAGMA} one calculates (see
Remark \ref{calc}) that
$$\mathrm{val}_{5}(L^{\mathrm{int}}(0,\phi))\geq 1.$$
Assuming that the 5-valuation is exactly 1
(see Remark \ref{calc} explaining the computational issues involved) 
this 
is enough to satisfy condition (2) of Theorem \ref{cor5.8} (cf. Remarks
\ref{OtoO'} and \ref{better}).
The characters
$\chi_0=\ov{\phi_{\fp}\epsilon}$ and $\chi_0^{-1}$
are $\Sigma$-admissible for appropriate sets $\Sigma$ (i.e. they
satisfy conditions (1), (3), (4) and (5) of Definition \ref{adm})
because the ray class field of conductor $5$ (a degree 16 extension
over $F$) has class number $3$ (as calculated by MAGMA assuming
GRH). Here we use that the splitting field $F(\chi_0)$ is contained
in the ray class field of $F$ of conductor $5$.
\end{example}

\begin{rem} \label{calc}
In our calculation above we use an operation in MAGMA called LRatio which
calculates a rational normalisation of the $L$-value of a modular form 
using
modular symbols. This calculation gives 5-valuation equal to 1. Because 
of  
the different period used by MAGMA we can only confirm that this 
provides a 
lower bound on the 5-valuation 
of $L^{\tuint}(0,\phi)=L(0,\phi)/\Omega^2$, for $\Omega$ the Neron
period of a suitable elliptic curve with complex multiplication by
$F$ (see e.g. \cite{Finis06}, p. 768). This follows from the following
relations between the different periods:

1. The proof of Lemma 7.1 of \cite{Dummigan03} shows that the period
used by MAGMA (RealVolume) is an integral multiple of the canonical
period $\Omega(f)^+$ defined by Vatsal \cite{Vatsal99} (up to
divisors of $Nk!$ for the level $N=51$ and weight $k=3$ of the
modular form).

2. Vatsal \cite{Vatsal99} proves that one can find a Dirichlet
character $\chi$ such that $\tau(\overline \chi) \cdot
\frac{L(1,f,\chi)}{(-2\pi i) \Omega(f)^\pm}$ (with
$\chi(-1)=(-1)^\pm$) is a $5$-unit. Note that Vatsal's condition
that $\overline \rho_f$ is absolutely irreducible is satisfied in
our case and $\Omega(f)^- \sim \Omega(f)^+$ because $f$ is a CM
form. Here we write $\sim$ to indicate equivalence up to $5$-unit .
Because $\pi L(1,f,\chi) \sim L(0,\phi \cdot {\rm
res}_F^{\bfQ}(\overline \chi))$ this implies that $\pi^2 \cdot
\Omega(f)^+$ is a $5$-integral multiple of $\Omega^2$.
\end{rem}

\subsection{A reducible deformation of $\rho_0$} \label{A reducible
deformation of rho_0}

Let $\Psi = \phi_{\fp} \epsilon$. Then $\chi_0 = \ov{\Psi}$.
For a finite set
of primes $S$ of $F$, let $L_{\Psi}(S)$ denote the maximal abelian
pro-$p$
extension of $F(\Psi)$ unramified outside $S$ and such that
$\Gal(F(\Psi)/F)$ acts on $\Gal(L_{\Psi}(S)/F(\Psi))$ via
$\Psi^{-1}$.

\begin{prop} \label{torsion1} The group 
$\Gal(L_{\Psi}(\Sigma\setminus
\{\ov{\fp}\})/F(\Psi))$ is a torsion $\bfZ_p$-module. \end{prop}

\begin{proof} The $\Sigma$-admissibility of $\chi_0$ implies that the
extension $L_{\Psi}(\Sigma\setminus \{\ov{\fp}\})/F(\Psi)$ is
unramified away from the primes lying over $\fp$. Then the claim
follows from the Anticyclotomic Main Conjecture of Iwasawa Theory
for imaginary quadratic fields (see \cite{Tilouine89},
\cite{Rubin91}, \cite{MazurTilouine90})  after noting that $L(0,\phi)
\neq 0$.
\end{proof}

\begin{cor} \label{torsion2} There does not exist a $\Sigma$-minimal
reducible deformation of $\rho_0$ into $\GL_2(A)$ if $A$ is not a
torsion $\Oo$-algebra. \end{cor}

\begin{proof} As in Proposition \ref{nored} such a deformation would have to be
of the form \be \label{red43} \rho = \bmat 1 & * \\ & \Psi\emat.\ee
By ordinarity, one must also have $$\rho|_{I_{\ov{\fp}}} \cong \bmat
1 \\ & \Psi|_{I_{\ov{\fp}}} \emat, $$ which implies that the
upper shoulder $*$ in (\ref{red43}) corresponds to an
extension $L/F(\Psi)$ which is unramified away from primes lying
over $\fp$. Since $A$ is not a torsion $\bfZ_p$-module, this would
contradict Proposition \ref{torsion1}. \end{proof}

\begin{rem} In \cite{SkinnerWiles97} Skinner and Wiles prove an $R=T$
theorem
for deformations of a certain class of reducible
(non-semi-simple) residual representations of $G_{\bfQ}$ of the form 
$\bsmat 1 & *
\\ & \chi \esmat$ for $\chi: G_{\bfQ} \rightarrow \ov{\bfF}_p^{\times}$ a
continuous character. They apply the numerical criterion of
Wiles and Lenstra \cite{Lenstra95}, \cite{Wiles95} by first relating the 
size of 
the relevant
universal deformation ring to a special value of the $L$-function of
$\chi$. They achieve this by studying the Galois cohomology of $\ad
\rho$ for a $\Sigma$-minimal reducible deformation $\rho$ with
values in a characteristic zero $\bfZ_p$-algebra $\Oo$. Here
$\Sigma$ is a finite set of primes of $\bfQ$ satisfying similar
conditions to the ones we imposed on our sets $\Sigma$. Corollary
\ref{torsion2} means that their method cannot be applied in our
case.
\end{rem}

Even though no $\Sigma$-minimal characteristic zero deformations of
$\rho_0$ exist, we now show that if one drops the ordinarity
condition at $\ov{\fp}$, it is possible to construct a reducible
(non-ordinary) deformation of $\rho_0$ into $\GL_2(\Oo)$.

\begin{prop} \label{red3} The group
$\Gal(L_{\Psi}(\Sigma)/F(\Psi))\otimes_{\Oo[\Gal(F(\Psi)/F)]}\Oo$ is an
$\Oo$-module of rank one. \end{prop}

\begin{proof} This follows from a result of Greenberg \cite{Greenberg78}
as we now explain. As before the $\Sigma$-admissibility of $\chi_0$ easily
implies that the extension $ L_{\Psi}(\Sigma)/F(\Psi)$ is unramified away
from primes lying over $\{\fp, \ov{\fp}\}$. Hence without loss of
generality we assume that $\Sigma = \{\fp, \ov{\fp}\}$. We have the
following diagram of fields
\be\xymatrix@!C{& F(\Psi) \ar@{-}[dl]_{\Gamma} \ar@{-}[dr]^{\Delta} \\
F_0 \ar@{-}[dr]_{\Delta} && F_{\iy} \ar@{-}[dl]^{\Gamma} \\
& F}\ee where $\Gamma\cong \bfZ_p$ and $\Delta$ is a finite group
(whose non-$p$-part is isomorphic to the group $\Gal(F(\chi_0)/F)$). Set
$X_{\Psi}:= \Gal(L_{\Psi}(\Sigma)/F(\Psi))$. Let $L/F(\Psi)$ be the
maximal abelian pro-$p$ extension of $F(\Psi)$ unramified away from
$\{ \fp, \ov{\fp}\}$ and write $X$ for $\Gal(L/F(\Psi))$. Then
$X_{\Psi}$ is a quotient of $X$. Both, $X$ and $X_{\Psi}$ are
$\bfZ_p[[\Gamma]]$-modules in a natural way. By choosing a generator
$\g$ of $\G$ we can make the indentification $$\Lambda:=
\bfZ_p[[\G]] \cong \bfZ_p[[T]]$$ by sending $\g$ to $T+1$. By
Theorem on page 85 of \cite{Greenberg78} we have
$$X\otimes_{\bfZ_p}E \cong
\Lambda_E^{\# \Delta} \oplus \textup{($\Lambda_E$-torsion)},$$ where
$\Lambda_E:= \Lambda \otimes E$.
Let $X_E:= (X\otimes E)$/($\Lambda_E$-torsion).

Consider the action of
$E[\Delta]$ on $X_E$. Let $\Delta^{\vee}$
denote the group of characters of $\Delta$, and
write
$$X_E=
\bigoplus_{\psi \in \Delta^{\vee}} X_E^{\psi},$$
where $$X_E^{\psi}:=\{x \in X_E\mid \sigma x =
\psi(\sigma)x \hf \textup{for every} \hs \sigma \in \Delta\}.$$
It is not hard to see that for
every $\psi \in \Delta^{\vee}$, one has $X_E^{\psi}
\neq 0$. This in particular means that every character of $\Delta$
appears exactly once, because the $\Lambda_E$-rank of $X_E$ equals
$\# \Delta$. Now,
consider the action of $\G$ on
$X_E$. Let $\Psi_0:= \Psi|_{\Gal(F(\Psi)/F_0)}$. Since
$\Gal(F(\Psi)/F) \cong \Delta \times \Gal(F(\Psi)/F_0)$,
we can study the action of the two direct summands
separately. We have $X_E^{\Psi|_{\Delta}} = \Lambda_E$, hence
$$\Gal(L_{\Psi}(\Sigma)/F(\Psi))\otimes_{\bfZ_p[\Gal(F(\Psi)/F)]} E \cong
\Lambda_E/(T+1-\Psi(\gamma)) \cong E,$$
where $\gamma$ is a topological
generator of $\Gal(F(\Psi)/F_0)$. This clearly implies the claim of the
proposition.
\end{proof}

 \begin{cor} \label{red5} There exists a deformation $\rho: 
G_{\Sigma}
\rightarrow \GL_2(\Oo)$ of $\rho_0$ of the form $$\rho \cong \bmat
1&
*\\ & \Psi \emat.$$ The extension $F(\rho)/F(\Psi)$ is unramified
away from $\{\fp, \ov{\fp}\}$. \end{cor}

\begin{proof} This follows easily from Proposition \ref{red3}. See for
example the discussion on page 10522 of \cite{SkinnerWiles97}.
\end{proof}

\begin{rem} The representation $\rho$ in Corollary \ref{red5} is not
ordinary. Indeed, if it were ordinary the representation
$\rho|_{D_{\ov{\fp}}}$ would have an unramified quotient. Since it
clearly has an unramified submodule, it would be split and thus the
upper shoulder $*$ would correspond to a non-$\bfZ_p$-torsion
extension of $F(\Psi)$ unramified away from $\fp$, which does not
exist by Proposition \ref{torsion1}. On the other hand $\rho$ is
\textit{nearly ordinary} in the sense of Tilouine (see e.g.
Definition 3.1 of \cite{Weston05}) with respect to the
upper-triangular Borels at $\fp$ and $\ov{\fp}$. Since one has
$$\rho|_{I_{\ov{\fp}}} \cong \bmat 1 &
*\\ & \epsilon|_{ I_{\ov{\fp}}}\emat,$$ the representation $\rho$
is, however, not  de Rham. \end{rem}

\section{Acknowledgements} The authors would like to thank Adebisi 
Agboola, 
Trevor Arnold, Frank
Calegari, John
Coates, Matthew
Emerton, Ralph Greenberg, Chris Skinner, and Jacques Tilouine for
helpful discussions and comments.

\bibliographystyle{amsalpha}
\bibliography{standard2}

\end{document}

%% file: settings.tex
\usepackage[all]{xy}
\usepackage{amsthm}
\SelectTips{cm}{10}\UseTips
\theoremstyle{plain}
\newtheorem{thm}{Theorem}
\newtheorem{prop}[thm]{Proposition}
\newtheorem{cor}[thm]{Corollary}
\newtheorem{lemma}[thm]{Lemma}

\theoremstyle{definition}
\newtheorem{definition}[thm]{Definition}

\newtheorem{example}[thm]{Example}
\newtheorem{rem}[thm]{Remark}

\numberwithin{thm}{section}
\numberwithin{equation}{section}